\documentclass[11pt,a4paper]{amsart}

\usepackage{amssymb}

\renewcommand{\epsilon}{\varepsilon}
\renewcommand{\setminus}{\smallsetminus}
\renewcommand{\emptyset}{\varnothing}

\newtheorem{theorem}{Theorem}[section]
\newtheorem{proposition}[theorem]{Proposition}
\newtheorem{corollary}[theorem]{Corollary}
\newtheorem{lemma}[theorem]{Lemma}
\newtheorem{conjecture}[theorem]{Conjecture}

\theoremstyle{definition}

\theoremstyle{remark}

\newcommand{\normal}{\lhd}

\newcommand{\Z}{\mathbb Z}

\newcommand{\R}{\mathbb R}

\newcommand{\BR}{\mathbb{R}}
\newcommand{\BQ}{\mathbb{Q}}
\newcommand{\BF}{\mathbb{F}}
\newcommand{\BZ}{\mathbb{Z}}
\newcommand{\BN}{\mathbb{N}}

\newcommand{\fpinfty}{{\FP}_{\infty}}

\newcommand{\FP}{\operatorname{FP}}

\newcommand{\cohom}[3]{H^{{\raise1pt\hbox{$\scriptstyle#1$}}}(#2\>\!,#3)}
\newcommand{\tatecohom}[3]%
  {\widehat H^{{\raise1pt\hbox{$\scriptstyle#1$}}}(#2\>\!,#3)}

\newcommand{\Cohom}[3]%
  {H^{{\raise1pt\hbox{$\scriptstyle#1$}}}\big(#2\>\!,#3\big)}
\newcommand{\Tatecohom}[3]%
  {\widehat H^{{\raise1pt\hbox{$\scriptstyle#1$}}}\big(#2\>\!,#3\big)}

\newcommand{\homol}[3]{H_{{\lower1pt\hbox{$\scriptstyle#1$}}}(#2\>\!,#3)}
\newcommand{\homolog}[2]{H_{{\lower1pt\hbox{$\scriptstyle#1$}}}(#2)}

\newcommand{\mono}{\rightarrowtail}
\newcommand{\epi}{\twoheadrightarrow}

\newcommand{\eg}{{\underline EG}}

\title{Centralisers of finite subgroups in soluble groups of type $\FP_n$}
\author{D.~ H. ~Kochloukova}
\address{Dessislava H.~Kochloukova, Department of Mathematics, University of Campinas, Cx. P. 6065,
13083-970 Campinas, SP, Brazil}
\email{desi@unicamp.br}
\author{C.~Mart\'inez-P\'erez}
\address{Conchita Mart\'inez-P\'erez, IUMA. Departamento de Matem\'aticas, Universidad de Zaragoza,
50009 Zaragoza, Spain} \email{conmar@unizar.es}
\author{B.~ E.~A.~Nucinkis}
\address{Brita E.~A.~Nucinkis, School of Mathematics, University of Southampton, Southampton,
SO17 1BJ, United Kingdom}
\email{bean@soton.ac.uk}
\date{\today} 

\thanks{This work was supported by EPSRC grant EP/F045395/1  and  LMS Scheme 4 grant 4708.
The second named author was also supported by Gobierno de Aragon and
MTM2007-68010-C03-01. The first named author is partially supported by CNPq, Brazil.}

\begin{document}

\maketitle

\thispagestyle{empty}

\begin{abstract} We show that for soluble groups of type $\FP_n,$ centralisers of finite subgroups need not be of type $\FP_n$. \end{abstract}


\section{Introduction}

\noindent
We study stabilizers of finite groups acting on soluble groups of type $\FP_n$. Our interest in this problem  derives from the study  of groups $G$ of type Bredon-$\FP_{d}$  with respect to the family ${\mathcal{F}}$ of all finite subgroups of $G$.   For the family $\mathcal F$ of finite subgroups, the classifying space for proper actions,  denoted by $\eg$,  has been widely studied. In particular, L\"uck  has shown \cite{Lueck} that $G$ admits a cocompact model for  $\eg$  if and only if $G$ has finitely many  conjugacy classes of finite subgroups and the stabilizer in $G$ of every finite subgroup is finitely presented and of type $\FP_{\infty}$.  Bredon cohomology with respect to the family of all finite subgroups can be viewed as the algebraic mirror to classifying spaces for proper actions. Its properties are similar to those of ordinary cohomology, the mirror for Eilenberg MacLane spaces.
The following result is an algebraic version of L\"uck's result, generalised to arbitrary $d$, and serves as the main motivation for this paper.

\begin{theorem}\cite[Lemma 3.1]{kmn}\label{bredonfpn}
A group is of type Bredon-$\FP_n$ if and only if it has finitely many conjugacy classes of finite subgroups and centralisers of finite subgroups are of type $\FP_n$.
\end{theorem}

\noindent
Recently it was proved that every virtually soluble group of type $\FP_{\infty}$ is of type Bredon-$\FP_{\infty}$ \cite{martineznucinkis}. Here we show that the equivalent statement does not hold for type $\FP_n$ using methods from $\Sigma$-theory developed by R. Bieri, J. Groves, R. Strebel and others. Even virtually metabelian groups  of type $\FP_n$ are not necessarily of type Bredon-$\FP_n$.
In sections \ref{prime} and \ref{torsionfree}  we present two types of examples for which we calculate the homological type of the  centralizers of finite actions. 
In these examples $G$ is an extension of $A$ by $Q$ where $A$ and $Q$ are abelian groups with $Q$ of rank $n$, $G$ is of type $\FP_n$ but not $\FP_{n+1}$  and there is a  finite group $H$ of order $n$ acting on $G$. Furthermore this $A$ has  Krull dimension 1 as a $\BZ Q$-module.
We show that if $H_0$ is a subgroup of index $d$ in $H$ then $C_G(H_0)$ is of type $\FP_d$ but not $\FP_{d+1}$. In the case when $A$ is of prime exponent, the finite group $H$ of our examples is cyclic. But in the case when $A$ is torsion-free, $H$ can be any finite group which is realisable as a Galois group of a finite extension over $\BQ$.  In particular $H$ can be any symmetric group. 

\noindent
In addition to these  examples
we prove some positive results describing the finiteness conditions of centralisers of finite subgroups. In particular, we show in Theorem \ref{metabelianbredonfpn} that a metabelian-by-finite group $G$ of type $\FP_n$ with finite Pr\"ufer rank  is of
type Bredon $\FP_d$ for $d=\lfloor \frac{n}{s} \rfloor,$ where $s$ denotes the upper bound on the orders of the finite subgroups of $G$ (see \ref{conjclasses}). Some partial results concerning Bredon homological type $\FP_k$ for soluble-by-finite groups of finite Pr\"ufer rank are included in section \ref{Prufer}.

\section{Preliminaries on the Bieri-Strebel Sigma invariant}

\noindent
Let $G$ be a finitely generated group. The character sphere $S(G)$ of $G$ is defined by $$S(G) = (Hom_{\BZ}(G, \BR) \setminus \{ 0 \}) / \sim,$$ where $\BR$ is considered as a group via addition and $\sim$ is the equivalence relation where $\chi_1 \sim \chi_2$ if there is a positive real number $r$ such that $r \chi_1 = \chi_2$. We write $[\chi]$ for the class of $\chi$ in $S(G)$ and denote 
 $$G_{\chi}=\{g\in G | \chi(g)\geq 0 \}.$$ Note that $G_{\chi}$ is a submonoid of $G$.
 
 \noindent
 Let $Q$ be a finitely generated abelian group and $A$ be a  $\BZ Q$-module. The Bieri-Strebel invariant \cite{BieriStrebel} is defined as
 $$
 \Sigma_A(Q) = \{ [\chi] \in S(Q) | A \hbox{ is finitely generated over } \BZ Q_{\chi} \}.
 $$
Complements in the character sphere are denoted as follows:
$$
\Sigma_A^c(Q) = S(Q) \setminus \Sigma_A(Q).
$$
We say that $A$ is $m$-tame as a
$\BZ Q$-module if whenever $[\chi_1], \ldots , [\chi_m] \in \Sigma_A^c(Q)$ we have $\chi_1 + \ldots + \chi_m \not= 0$.
The following conjecture was suggested in \cite{BieriGroves} after R. Bieri and R. Strebel had already resolved the case $m= 2$ \cite{BieriStrebel}.

\medskip\noindent
{\bf The $\FP_m$-Conjecture.} {\it Let $ 1 \to A \to G \to Q \to 1$ be a short exact sequence of groups with $G$ finitely generated and $A$ and $Q$ abelian. Then $G$ is of type $FP_m$ if and only if $A$ is $m$-tame as a $\BZ Q$-module.}

\medskip\noindent
Though the $\FP_m$-Conjecture is still open in general, it was shown to hold for metabelian groups of finite Pr\"ufer rank \cite{aberg} (recall that a group is said to be of finite Pr\"ufer rank if  there is an upper bound on the number of generators of the finitely generated subgroups).

\section{Metabelian groups of finite Pr\"ufer rank}

\bigskip\noindent
Let $Q$ be a finitely generated abelian group acting 
on an abelian group $A$ which is finitely generated as a $\Z
Q$-module. If a group $H$ acts on both $Q$ and $A$ we say that the
actions are compatible if
$$(a^q)^h=(a^h)^{q^h}$$
for any $a\in A,$ $q\in Q$, $h\in H.$
Note that there is an induced action of $H$ on the valuation sphere such that
for $h\in H$, $[\upsilon^h]$ is given by
$$\upsilon^h (q):=\upsilon(q^{h^{-1}}).$$
If the actions of $H$ and $Q$ on $A$ are compatible then $\Sigma_A^c(Q)$ is $H$-invariant (see \cite[3.4]{martineznucinkis}).

\begin{lemma}\label{pos} Let $1 \to A \to G \to Q \to 1$ be a short exact sequence of groups with $G$ finitely generated and $A$ and $Q$ abelian. Let $H$ be a finite group acting on $G$ such that $A$ is $H$-invariant.  Suppose $A$ is $n$-tame as a $\BZ Q$-module, where $Q$ acts on $A$ by conjugation and $d$ is a positive  integer such that
$$d|H|\leq n.$$ Then $A_0=C_A(H)$ is a finitely generated and $d$-tame  $\Z C_Q(H)$-module.
\end{lemma}
\begin{proof} 
Let $Q_0=C_Q(H)$. The finite group $H$ acts compatibly on $A$ and $Q$ and we can apply
 \cite[Lemma 3.5]{martineznucinkis} to see that there is a subgroup $Q_1$ of $Q$ such that $\widetilde{Q} = Q_0
\times Q_1$ is a subgroup of finite index in $Q$ and  \begin{equation} \label{annihilate} \sum_{t \in H} t \hbox{
acts trivially on  }Q_1. \end{equation} 

\noindent To prove that $A_0$ is a finitely generated $d$-tame $\BZ Q_0$-module it suffices to show that $A$ is a finitely generated $d$-tame $\BZ Q_0$-module, as $\Z Q_0$ is Noetherian and furthermore $d$-tameness is preserved by submodules \cite[Lemma 1.1]{BieriStrebel}.

\noindent We show first that $A$ is finitely generated
as a $\Z Q_0$-module
and  by Corollary \cite[Cor.~4.5]{bs81} this is equivalent to the following:
for every non-zero character $\chi : \widetilde{Q}  \to
\R $ such that $\chi(Q_0) = 0$ we have
that $[\chi] \in \Sigma_A(\widetilde{Q})$. Assume now that for some
$\chi$ as above we have that
$[\chi] \notin \Sigma_A(\widetilde{Q})$. Then using (\ref{annihilate}) $$\sum_{t \in H}[\chi ^
t] = 0 \hbox{
and }[\chi^t] \in \Sigma_A^c(\widetilde{Q}),$$ so $A$ is not $|H|$-tame, contradicting the hypothesis that $A$ is $n$-tame as a $\Z
Q$-module and therefore is also $n$-tame as a $\Z \widetilde{Q}$-module.

\noindent
Suppose that $A$ is not $d$-tame as a $\Z Q_0$-module. Then there are
elements \hfill\break $[\chi_1], \ldots, [\chi_d] \in \Sigma_{A_0}^c(Q_0)$ such that
$\sum_{1 \leq i \leq d} \chi_i = 0$. By \cite[Lemma 3.3]{martineznucinkis} there are
homomorphisms $\mu_i : Q_1 \to \R$ such that $[\chi_i + \mu_i] \in
\Sigma_A^c(\widetilde{Q})$. Using (\ref{annihilate}) we
obtain that $$\sum_{t \in H} \mu_i^t = 0 \hbox{ for every } 1 \leq i \leq d.$$
Hence
$$
\sum_{1 \leq i \leq d} \sum_{t \in H} (\chi_i^t + \mu_i^t)  = \sum_{1
\leq i \leq d} \sum_{t \in H} \chi_i^t  = (\sum_i \chi_i) | H |
= 0.
$$
Thus we have $d | H |$  characters  $\{ \chi_i^t + \mu_i^t \}_{1 \leq i \leq d, t \in H}$
summing to 0 and
$[\chi_i^t + \mu_i^t] \in \Sigma_A^c(\widetilde{Q})$, contradicting the fact that $A$
is $n$-tame as a $\Z \widetilde{Q}$-module and that  $d | H |\leq n$.  Hence $A$ is $d$-tame as a $\Z
Q_0$-module, proving the claim.
\end{proof}

\noindent Let $1 \to A \to G \to Q \to 1$ be a short exact sequence of groups with $G$  finitely generated and $A$ and $Q$ abelian. Consider the
following two conditions:

(i) $G$ has finite Pr\"ufer rank and $A = G'$;

(ii) $G=Q\ltimes A.$

\begin{proposition}\label{pos2} Assume $G$ satisfies either i) or ii) above and that
$A$ is $n$-tame as a $\Z Q$-module. Let $H$ be a finite group acting on
$G$ such that in case ii) $A$ and $Q$ are $H$-invariant and let $d$ be a positive integer with
$$d | H |\leq n.$$
Then  $C_A(H)$ is a finitely generated $d$-tame $\Z [C_G(H)/C_A(H)]$-module.
\end{proposition}
\begin{proof} By Lemma \ref{pos}, $A_0=C_A(H)$ is $d$-tame as a $\Z C_Q(H)$-module. In the
finite Pr\"ufer rank case it suffices to take into account that by
\cite[Lemma 3.10]{martineznucinkis}  the index $| C_Q(H):AC_G(H)| $
is finite, so $A_0$ is also $d$-tame as a $\Z [C_G(H)/C_A(H)]$-module.

Assume now that $G=Q\ltimes A.$ Then $C_G(H)=C_Q(H)\ltimes A_0$ so the result also follows.
\end{proof}

\noindent For a real number $r$ denote by $\lfloor {r}\rfloor$  the integral part of $r$. This  is the unique integer  $\lfloor {r}\rfloor$ such that $r - \lfloor {r} \rfloor \in [0, 1)$.

\begin{corollary}\label{fpncentralisers} Let $G$ be a metabelian group of type $\FP_n$ with finite Pr\"ufer rank. Let $H$ be a finite group acting on $G$. Then $C_G(H)$ is
of type $\FP_d$ for $d=\lfloor \frac{n}{| H |}\rfloor$. 
\end{corollary}

\begin{proof}
This follows from Proposition \ref{pos2} and the fact that the $\FP_m$-Conjecture holds for metabelian groups of finite Pr\"ufer rank \cite{aberg}.
\end{proof}

\begin{proposition}\label{conjclasses}
Let $G$ be a soluble-by-finite group of finite Pr\"ufer rank. Then there is a bound on the orders of the finite subgroups of $G$. Furthermore, $G$ has finitely many conjugacy classes of finite subgroups.
\end{proposition}

\begin{proof}   Let $G$ be a soluble group of finite Pr\"ufer rank. Then it is minimax; see  for example \cite[Exercise 14.1.4]{rob}. The proof of \cite[Theorem 10.33]{rob2} implies that $G$ has a bound on the orders of its finite subgroups. Hence the argument of \cite[Theorem 2.4]{martineznucinkis} can be applied implying that $G$ has finitely many conjugacy classes of finite subgroups.
\end{proof}

\noindent The above results  give a lower bound for the Bredon-type of metabelian groups. This bound turns out to be the best possible as the examples in the next sections show.

\begin{theorem}\label{metabelianbredonfpn} Let $G$ be a metabelian-by-finite group of finite
Pr\"ufer rank  of type $\FP_n$.  Then $G$ is of
type Bredon $\FP_d$ for $d=\lfloor \frac{n}{s} \rfloor,$ where $s$ denotes the upper bound on the orders of the finite subgroups of $G$.
\end{theorem}

\begin{proof}   Apply Corollary \ref{fpncentralisers}, Proposition \ref{conjclasses} and  Theorem \ref{bredonfpn}.
\end{proof}

\section{Examples of virtually metabelian groups}
\subsection{An example of prime characteristic} \label{prime}

In this section we construct examples of metabelian groups $G = A \rtimes Q$ 
 where $Q$ is free abelian
of finite rank and $A$ an infinite abelian $p$-group,  which show that the bound in Corollary \ref{fpncentralisers} can be sharp.
More precisely, for a certain integer $n$, these groups are of type $\FP_{n}$  but
admit an action of a finite cyclic group $H_0$ such that $C_G(H_0)$ is not of type $\FP_{d+1}$ for $d=\lfloor \frac{n}{\mid H_0\mid}\rfloor$.

\noindent
Let $\BF_p$ be the field with $p$ elements and $\BF_p \subset k$ be
a finite field extension of degree $m$;   $k$ is a field with $p^m$
elements. Then the Galois group $Gal(k \mid \BF_p)$ is  a cyclic group generated
by  the M\"obius map $\sigma$ sending an element to its $p$th
power. By \cite[Ch.~VIII,~Sec.~12,~Thm.~20]{Langbook}
there is an element $a \in k$ such that the set $\{ \sigma^j(a) \}_{0 \leq
j \leq m-1}$ is linearly independent over $\BF_p$.  Hence $k = \BF_p(a)$. Note
that
\begin{equation} \label{eq1}
b = tr(a) =  \sum_{0 \leq j \leq m-1} \sigma^ j(a) \in \BF_p \setminus \{ 0 \}.
\end{equation}
By multiplying $a$ with a non-zero element of $\BF_p$ we can assume that $b
= 1$.  Define the set
$$\{ a_1 = a, a_{i+1} = \sigma(a_i)  + a \}_{i \in \BN}.$$

\begin{lemma}   \label{new} The set $\{a_j \}_{1 \leq j \leq
mp}$ contains $mp$ different elements and $a_m = 1, a_{pm} = 0$.
\end{lemma}

\begin{proof} Note that $a_m = b = 1$ and so $a_{m+i} = a_i + 1$ for every $i$. 
Then $a_{pm} = a_m + (p-1) = 1 + (p-1) = 0$.
Furthermore $a_1 = a, a_2, \ldots, a_m = 1$ are linearly independent over $\BF_p$ and thus 
$$\{a_j \}_{1 \leq j \leq
mp} = \cup_{1 \leq i \leq m} \{ a_i + \BF_p \}$$
 contains   $mp$ different elements.
\end{proof}

\noindent
Define $A$ to be the localization $k[x_1,  1/ (x_1 + a_{j})]_{1 \leq j \leq
mp}$ of the polynomial ring $k[x_1]$. Let $H$ be a cyclic group of order
$pm$ with a generator $\mu$  acting on the  field $k(x_1)$ in
the following way : the restriction of $\mu$ to the field $k$ is the
M\"obius map $\sigma$, and $\mu(x_1) = x_1 +   a$. Thus $\mu^i(x_1) = x_1 +
a_i$. In particular $\mu^m(x_1) = x_1 + 1$ and by Lemma \ref{new}  $\mu$ has order $pm$.
Moreover we get an induced action of $H$ on $A$.

\noindent
Consider the split extension $G = A \rtimes Q$, where $Q$ is a free abelian
group with generators
$q_1, \ldots, q_{pm}$ and $q_i$ acts on $A$ by conjugation as multiplication with  $x_1 +
a_i$.
The generator $\mu$ of the cyclic group $H$ acts on $Q$ by sending $q_i$
to $q_{i+1}$, where $q_{pm+1} = q_1$. This, together with the above action
of $H$ on $A$ induces an action of $H$ on $G$.

\begin{lemma} \label{tame1} The $\BZ Q$-module $A$ is $(pm)$-tame but not $(pm+1)$-tame. Hence $G$ is of type $FP_{pm}$ but not of type $FP_{pm+1}$.
\end{lemma}

\begin{proof} Note first that  $A$ is a cyclic $\BZ Q$-module, where $Q$
acts via conjugation. So $A$ is a quotient ring of $\BZ Q$ and the
embedding of $Q$ in $\BZ Q$ induces an embedding of $Q$ in $A$. 

Let $\chi : Q \to \BR$ be a non-zero character such that
$[\chi] \in \Sigma_A^c(Q)$. Then by \cite[Thm.~8.1]{BieriGroves1} and the fact that $A$ is a cyclic $\BZ Q$-module there is a real
valuation $v : A \to \BR \cup \{ \infty \}$ such that the restriction of
$v$ to $Q$  extends $\chi$. 
\noindent
Since
$k$ is finite $v(k \setminus 0) = 0$ and so if $v$ has a positive value on
one of the elements of the basis $Y = \{ x_1 + a_j \}_{1 \leq j \leq pm}$ of
$Q$ as embedded in $A$, then $v$ has zero value on  all other
elements of the same basis. 

If $v$ has a negative value on one of the
elements of $Y$ then the $v$-value on all elements of $Y$ is the same. Thus $\Sigma_A^c(Q)$ has exactly $pm + 1$ points, $A$ is
$(pm)$-tame but is not $(pm + 1)$-tame as a $\BZ Q$-module.

Finally note that the $\FP_n$-Conjecture for metabelian groups holds for groups which are extensions of $A$ by $Q$, where  $Q$ is a finitely generated abelian group and $A$ is a finitely generated $\BZ Q$-module of Krull dimension 1 and of finite exponent as an additive group \cite[Cor.~C]{Koch}.
\end{proof}

\noindent
Define $H_0$ to be the subgroup of $H$ generated by $\mu^{d}$ for some
positive divisor $d$ of $m$. Thus $[H : H_0] = d$.
We want to study the centralizer $C_G(H_0)$ of $H_0$ in $G$. Note that
$C_G(H_0) = C_A(H_0) \rtimes C_Q(H_0)$. We write $A_0$ for $C_A(H_0)$ and
$Q_0$ for $C_Q(H_0)$ and consider $A_0$ as a $\BZ Q_0$-module via
conjugation. Note that $A_0 \not= 0$ since $\prod_{h \in H} h(x_1) $ is an element of $A_0$ and a non-zero polynomial of degree
$pm$.
Since the generator $\mu$ of $H$ cyclicly permutes the basis $q_1, \ldots, q_{pm}$ of $Q$,  $Q_0$ is a free abelian group of rank $d$.

\begin{lemma} \label{dplus1} The $\BZ Q_0$-module $A_0$ is not
$(d+1)$-tame.
\end{lemma}

\begin{proof}
Let $T_0$ be the image of $Q_0$ under the embedding of $Q$ in $A$.
Then there is no bound on the degree, as polynomials in the variable
$x_1$,  of the elements of $T_0 \prod_{h \in H} h(x_1) $  and    $T_0
\prod_{h \in H} h(x_1) \subseteq A_0$. In particular $A_0$ is not finite
dimensional.

\noindent Suppose now that $A_0$ is $(d+1)$-tame as a $\BZ Q_0$-module. 
Note that by Proposition \ref{pos2} 
$A_0$ is finitely generated as a $\BZ Q_0$-module.
Since $d$ is
the rank of $Q_0$ we get that $A_0$ is $\infty$-tame as a $\BZ
Q_0$-module. Then  the main result of \cite{BieriGroves2} implies that for all $s$ the  tensor powers $\otimes^s 
A_0$ are finitely generated as  $\BZ Q_0$-modules with diagonal $Q_0$-action. Hence they are also finitely generated as  $\BF_p Q_0$-modules.  In particular, the Krull dimension of  $\otimes^s  A_0$  is at
most the Krull dimension of  $\BF_p Q_0$, which is the rank of $Q_0$. On
the other hand the Krull dimension of $\otimes^s  A_0$  is at least $s$
times the Krull dimension of $A_0$ \cite[Lemma~1]{Brookes} and the Krull dimension of $A_0$ is not
zero since $A_0$ is infinite dimensional. Thus $s$ cannot be arbitrarily
large, giving a contradiction.

\end{proof}

\begin{corollary} The group $C_G(H_0)$ is of type $\FP_d$ but not of type $\FP_{d+1}$.
\end{corollary}

\begin{proof} 
By Proposition \ref{pos2} and Lemma \ref{dplus1} $A_0$ is $d$-tame but not $(d+1)$-tame as a $\BZ Q_0$-module.
The claim now follows from the fact that the $\FP_m$-Conjecture for metabelian groups holds for groups which are extensions of $A_0$ by $Q_0$, where  $Q_0$ is a finitely generated abelian group and $A_0$ is a finitely generated $\BZ Q_0$-module of Krull dimension 1 and of finite exponent as an additive group \cite[Cor.~C]{Koch}.
\end{proof}

\subsection{An example of finite Pr\"ufer rank} \label{torsionfree}
The group $A$ constructed in the previous section was of infinite Pr\"ufer rank.  We construct  examples similar to the above, but
now $A$ will be a torsion free group of finite Pr\"ufer rank. Moreover, the finite group $H_0$ acting on $G$ will be a subgroup
of the Galois group of a finite extension of $\BQ$. Note also that any finite group is a subgroup of some symmetric group and any symmetric group is the Galois group of some extension of $\BQ$.

Let $K:\BQ$ be a Galois extension
with Galois group $H$. Choose a primitive element $\zeta$ with
minimal polynomial $f(x),$ note that $\zeta$ can chosen to be integral
over $\BZ$ so we may assume $f(x)\in\BZ[x].$ Let $O_K$ be the
integral closure of $\BZ$ in $K$ and $d_K$ its discriminant.

By \cite[Lemma~4.1]{DesiBritaConchita} there are infinitely many primes $q$ for which there exists some
integer $k_q$ with $q|f(k_q)$ so we may take
$q,k:=k_q$ such that 
$$q|f(k), q\nmid |O_K/\BZ[\zeta]|  \hbox{ and }q\nmid d_K.$$
By Dedekind's Theorem the last condition implies that $q$ is not ramified in $O_K$.
Moreover, the extension $K:\BQ$ is Galois and therefore
$$qO_K=I_1\ldots I_{t_0}$$
for distinct prime ideals $I_i$ which form a single
$H$-orbit. In particular, $O_K/qO_K$ is a product of fields. Put $I=I_1$. Then for any $i$
$$\delta=[O_K/I:\BF_q]=[O_K/I_i:\BF_q],$$
and $n=[K:\BQ]= \delta t_0$. 

\noindent
The condition
$q\nmid |O_K/\BZ[\zeta]|$ implies
$$O_K/qO_K\cong\BZ[\zeta]/q\BZ[\zeta]=\BZ[x]/(q,f(x))=
\BF_q[x]/(\bar f(x)).$$ 
As $x-\bar k\mid \bar f(x)$ in $\BF_q[x]$, we deduce that
$\delta = 1$. Therefore $n= t_0$ and 
$$qO_K=\prod_{t\in H}I^t$$ with $I\unlhd O_K$ prime and the $I^t$ are all
distinct.

\noindent
Let $r$ be the ideal class number of the Dedekind domain $O_K$. Then
$$I^r=\alpha O_K$$ for some $\alpha\in O_K.$

\noindent
Take any natural prime number $p\neq q$. Then $$pO_K=J_1^s\ldots J_t^s$$
where the $J_i$ are prime ideals which form a single $H$-orbit.

\bigskip\noindent
Now, let $Q=\langle q_t \rangle_{t \in H}$ be free abelian of rank
$n=| H | $ and let
$$A=O_K[\alpha^t/p,p/\alpha^t]_{t\in H}.$$ 
Let
each $q_t$ act on $A$ by multiplication by $\alpha^t/p.$ The group
$H$ acts on
 $Q$  by $q_t^{t_1} = q_{t t_1}$, where $t, t_1 \in H$. The group $H$ also acts on $K,$ and $A$ is invariant under this action.
 Moreover the actions of $Q$ and $H$ on $A$ are compatible so this gives an action of $H$ on the group $G:=Q\ltimes A.$

\begin{proposition}\label{fpn} The group $G$ is of type $\FP_n$ and not  of type $\FP_{n+1}.$
\end{proposition}
\begin{proof}
Consider 
the homomorphism of groups
$$\tau:Q\to A$$ given by $\tau(q_h)=\alpha^h/p$. This map can be extended to an epimorphism
$$\tau:O_K Q\to A,$$
where $O_K Q$ is  the group algebra of $Q$ with coefficients in $O_K.$
As $A$ is a domain, $A\cong O_K Q/P$ for some prime ideal $P\normal O_K Q$. Note also that $P\cap O_K=0$.  Thus, by \cite[2.4]{bs81} 
with $R=O_K$, the set $\Sigma_A^c(Q)$ is finite and discrete. Therefore \cite[2.2~and~2.1]{bs81} imply that$$\Sigma_A^c(Q) =\{[\upsilon\circ\tau]:\upsilon\circ\tau\neq 0\}$$
where $\upsilon:K\to\BZ_\infty$ is a discrete valuation of $K$.
Note that any real valuation  of $\BZ$ is non-negative $\upsilon(O_K)\geq 0.$
The discrete valuations of $K,$ up to multiplication with a positive real number, are the $P$-adic valuations
$\upsilon_P$ where $P\unlhd O_K$ is a prime ideal. The valuation
$\upsilon_P$ is given by
$$\upsilon_P(a)=m$$
such that
$$aO_K=P^mM$$
where $M$ is a fractionary ideal such that $P$ does not appear in its decomposition in primes. Note that for $x \in H$
$$(\alpha^x/p) O_K=(I^x)^r(pO_K)^{-1}$$
and therefore for $h,x\in H$
$$\upsilon_{I^h}(\alpha^x/p)=\Bigg\{\begin{aligned}
&r\text{ if }x=h\\
&0\text{ otherwise.}\\
\end{aligned}$$
Also, for any  prime ideal $J$ of $O_K$ with $J|pO_K$  there is some $s_0 >0$ with
$$\upsilon_{J}(\alpha^x/p)=- s_0.$$
Therefore 
$\Sigma_A^c(Q)$ has exactly $(n+1)$ points and the $\BZ Q$-module $A$ is $n$-tame but not $n+1$-tame.
By \cite{aberg} the $\FP_m$-Conjecture holds for metabelian groups of finite Pr\"ufer rank, so the group $G$ is of type $\FP_n$ but not $\FP_{n+1}$.  
\end{proof}

\noindent
Now let $H_0 \leq H$ be a subgroup of index $d$. Note that $H_0$ acts on $G$ and
$$d | H_0 | = n.$$

\noindent
Moreover $C_G(H_0)=Q_0\ltimes A_0$ with $Q_0=C_Q(H_0)$ and  $A_0=C_A(H_0)$. 
The group $Q_0$ is free abelian of rank $d$ with basis
$$\{s_x:=\prod_{t\in H_0} q_{x t}:x\in H / H_0\}.$$

\begin{proposition} The group $C_{G}(H_0)$ is of type
$\FP_d$ but not of type $\FP_{d+1}$.
\end{proposition}
\begin{proof} Note first that Lemma \ref{pos2} and the fact that the $\FP_m$-Conjecture for groups of
finite Pr\"ufer rank holds, see \cite{aberg}, imply that $C_G(H_0)$ is of type $\FP_d$. In particular $A_0$ is a finitely generated  $\BZ Q_0$-module. 

\noindent
Suppose that $C_G(H_0)$ is of type $\FP_{d+1}$. Since the $\FP_m$-Conjecture holds for groups of
finite Pr\"ufer rank, we deduce that $A_0$ is $(d+1)$-tame as a $\BZ Q_0$-module: every $d+1$ elements of $\Sigma_{A_0}^c(Q_0)$ lie in an open hemisphere. Since $d+1 > rk(Q_0)$ we get that $A_0$ is $\infty$-tame as a $\BZ Q_0$-module.

\noindent
Let $\chi : Q \to \BR$ be a non-zero homomorphism with $\chi(Q_0) \not= 0$. We claim:
\begin{equation} \label{sexta} \hbox{ If }[\chi] \in \Sigma_A^c(Q)\hbox{ then }[\chi \mid_{Q_0}] \in \Sigma^c_{A_0}(Q_0). \end{equation}
Indeed, $A$ is a cyclic $O_K Q$-module. So $A \simeq O_K Q / I$ for $I = ann_{O_K Q}(A)$. Since $[\chi] \in \Sigma_A^c(Q)$ and $O_K$ is finitely generated as a $\BZ$-module  one sees that $A$ is not finitely generated as a $O_K Q_{\chi}$-module. Then by applying  \cite[Thm.~8.1]{BieriGroves1}, $\chi$ lifts to a real valuation $v : A \to \BR \cup {\infty}$ such that $v(O_K) \geq 0$. 
The
restriction $$w = (v \mid_{A_0}) : A_0 \to \BR \cup {\infty}$$ has the following property:
$w( a q) = w(a) + \chi(q)$ for every $a \in A_0, q \in Q_0$. This shows that, provided $w(A_0) \not= \infty$, $A_0$ is not finitely generated as $O_K Q_{\chi}$-module and consequently  $[\chi \mid_{Q_0}] \in \Sigma^c_{A_0}(Q_0)$. Note that $v^{-1}(\infty)$ is a prime ideal in $A$ and that $A$ has Krull dimension 1. Thus $v^{-1}(\infty) = 0$.

\noindent
Finally, since $A_0$ is an $\infty$-tame  $\BZ Q_0$-module, (\ref{sexta}) implies  that $\Sigma_A^c(Q)$ lies in a closed half subsphere of $S(Q)$. This contradicts the description of $\Sigma_A^c(Q)$ in the proof of  
 Lemma \ref{fpn}.
\end{proof}

\section{Soluble groups of finite Pr\"ufer rank} \label{Prufer}

\noindent In this section we shall consider soluble groups of finite Pr\"ufer rank  of type $\FP_n$ and  cohomological finiteness conditions for centralisers of finite subgroups.
Soluble groups $G$ of finite Pr\"ufer rank are nilpotent-by-abelian-by-finite, which follows from the proof of \cite[Theorem 10.38]{rob2}. Then $G$ has a subgroup of finite index $G_1$ that is characteristic and is nilpotent-by-abelian.
For such  a  group $G$ there is a group extension:

\begin{equation}\label{one}
N \mono G \epi Q   
\end{equation}  
where $N$ is the commutator of  $G_1$, $N$ is nilpotent and $Q$ is abelian-by-finite. 

\medskip\noindent  For soluble groups of finite Pr\"ufer rank there is as of now no result like Theorem \ref{metabelianbredonfpn}. 
We have, however, some partial results which suggest a general conjecture for Bredon type $\FP_m$:

\begin{conjecture}\label{Bredonfpdconj} Let $G$ be a soluble group of finite Pr\"ufer rank and of type $\FP_n$. Then $G$ is of type Bredon-$\FP_m$ for $m= {\lfloor \frac{n}{sc}\rfloor}$, where $s$ is the bound on the orders of the finite subgroups of $G$ and $c$ is the nilpotency class of $N$ in (\ref{one}) above.
\end{conjecture}

\noindent
Note that by Proposition \ref{conjclasses}, $G$ has finitely many conjugacy classes of finite subgroups. Therefore  Theorem \ref{bredonfpn}  implies that $G$ is of type Bredon-$\FP_m$ if and only if  the centralizer in $G$ of any finite group of automorphisms of $G$ is of type $\FP_m$.
The $\FP_m$ property is invariant under finite group extensions and finite index subgroups.  In particular, let  $M_1 \leq M_2$ be groups such that  $M_1$ has finite index in $M_2$.  Then $M_1$ is of type $\FP_m$ if and only  if $M_2$ is of type $\FP_m$. It therefore suffices to show that the centralizer in $G_1$ of any finite group $H$ of automorphisms of $G$ is of type $\FP_m$. Thus without loss of generality we can assume that $G = G_1$.  Hence from now on let  $Q$ be abelian, $G $ be nilpotent-by-abelian and let $H$ be a finite group of automorphisms of $G$. Note that we may also  assume that $m \geq 1$ in the above conjecture. Otherwise there is nothing to prove. Thus we can suppose that
\begin{equation} \label{ncs}
n \geq cs.
\end{equation}

\noindent
We denote by $A=N/N'$ the abelianization of $N$. We shall make use of the following result by  {\AA}berg:

\begin{proposition}\cite[IV.2.2]{aberg}  Let $G$ be a soluble group of finite Pr\"ufer rank   of type $\FP_n.$ Then $A$ is an $n$-tame   $\BZ Q$-module. 
\end{proposition}

\begin{lemma}\label{reduction} Let $G$ be a nilpotent-by-abelian group of finite Pr\"ufer rank and of type $\FP_n$. Let $H$ be a finite group acting on $G$. Let $N$ be as in (\ref{one}) above. Then 
$A=N/N'$ is finitely generated and $\lfloor \frac{n}{|H|} \rfloor$-tame as a $C_G(H)/C_N(H)$-module.

\end{lemma}

\begin{proof} As $G$ is of type $\FP_n$ and of finite Pr\"ufer rank,
$A$ is $n$-tame as a $Q$-module by \cite[Proposition IV 2.2]{aberg} . Using a
similar argument to Lemma \ref{pos}, we deduce that $A$ is finitely generated and $\lfloor \frac{n}{|H|} \rfloor$-tame  as a
$C/N$-module, where $C/N=C_Q(H)$. Moreover, the index $|C:C_G(H)N|$ is
finite \cite[Proposition 3.11]{martineznucinkis} (note that although the groups considered in \cite{martineznucinkis} are of type $\FP_\infty$, in this result the group only has to be finitely generated and of finite Pr\"ufer rank).
\end{proof}

\noindent
For a group $M$ denote by $\gamma_i(M)$ the $i$-th term of the lower central series of $M$.

\begin{theorem}\label{centralser}
Let $G$ be a nilpotent-by-abelian group of finite Pr\"ufer rank, of type $\FP_n$ and let $H$ be a finite group of order $s$ acting on $G$ such that (\ref{ncs}) holds. Let $N$ be as in (\ref{one}) above. Then $D = \oplus_i \gamma_i(C_N(H)) / \gamma_{i+1} (C_N(H))$ is finitely generated and  $\lfloor\frac{n}{sc}\rfloor$-tame as a $C_G(H)/C_N(H)$-module, where $c$ is the nilpotency class of $N$.
\end{theorem}

\begin{proof} Let $Q_0=C_G(H)/C_N(H)$ and let $d=\lfloor\frac{n}{s}\rfloor$. By Lemma \ref{reduction} $A=N/N'$ is $d$-tame as a $Q_0$-module. By (\ref{ncs}) $d \geq c$ and hence for $i \leq c$ we have that
$\gamma_i(N) / \gamma_{i+1} (N)$ is finitely generated as a $\BZ Q_0$-module, as it is a quotient of the finitely generated $\BZ Q_0$-module $\otimes^i A$. 

\noindent
Let $B$ be any abelian subsection of $N$ that is invariant under the action of $Q_0$ and let $A_i$ be the $i$-th factor of the
lower central series of $N$. Then $B$ has a series of $\BZ Q_0$-modules with factors $B_1,...,B_c$ such that each $B_i$ is a subsection of $A_i$ and $A_i$ is a subsection of $\otimes^i A$. Then all $A_i$ and $B_i$ are finitely generated as $\BZ Q_0$-modules and so $B$ is finitely generated as a $\BZ Q_0$-module.
By \cite[Lemma 1.1]{bs81}
$$\Sigma^c_B(Q_0) = \bigcup_{i=1}^c \Sigma^c_{B_i}(Q_0),$$
and for $i \leq c$
$$\Sigma_{B_i}^c(Q_0) \subseteq \Sigma_{\otimes^i A}^c(Q_0) \subseteq conv_{\leq i}\Sigma_A^c(Q_0) \subseteq conv_{\leq c}\Sigma_A^c(Q_0).$$
Therefore 
$$
\Sigma^c_B(Q_0) \subseteq conv_{\leq c}\Sigma_A^c(Q_0).$$
Now apply this for $B = \gamma_i(C_N(H)) / \gamma_{i+1} (C_N(H))$. Then $\gamma_i(C_N(H)) / \gamma_{i+1} (C_N(H))$ is finitely generated as a $\BZ Q_0$-module and using \cite[Lemma 1.1]{bs81} again 
$$
\Sigma^c_D(Q_0)  = \cup_i \Sigma_{\gamma_i(C_N(H)) / \gamma_{i+1} (C_N(H))}^c(Q_0) \subseteq conv_{\leq c}\Sigma_A^c(Q_0).$$
Therefore
if $0\notin conv_{\leq d}\Sigma_A^c(Q_0)$ then also $0\notin conv_{\leq \frac{d}{c}}\Sigma_{D}^c(Q_0)$ and $D$ is $\lfloor \frac{d}{c}\rfloor$-tame as a $\BZ Q_0$-module. 
\end{proof}

\begin{corollary}\label{cortrivialsub}  Let $G$ be a finitely generated nilpotent-by-abelian group of finite Pr\"ufer rank.  Let $N$ be nilpotent of class $c$ as in (\ref{one}) and assume that $G$ has type $\FP_n$ for some $n \geq c$. Then $\oplus_i \gamma_i(N) / \gamma_{i+1}(N)$ is finitely generated and $\lfloor \frac{n}{c} \rfloor$-tame as a $\BZ Q$-module.
\end{corollary}

\begin{proof} Take $H=\{1\}$ in Theorem \ref{centralser}.
\end{proof}

\noindent
The following result together with the remarks after Conjecture \ref{Bredonfpdconj} show that Conjecture \ref{Bredonfpdconj} holds for $m = 1$.

\begin{corollary}\label{centralfg}
Let $G$ be a finitely generated nilpotent-by-abelian group of finite Pr\"ufer rank.  Let $N$ be nilpotent of class $c$ as in (\ref{one}) with $Q$ abelian. Let $H$ be a finite group of order $s$ acting on $G$.  Assume also that $G$ is  of type $\FP_n$ such that $\frac{n}{cs} \geq 1$.   Then $C_G(H)$ is finitely generated.
\end{corollary}

\begin{proof} 
Since $N$ is a nilpotent group it suffices to show that 
$$B=(C_G(H)\cap N)/(C_G(H)\cap N)'$$ is finitely generated as a $Q_0$-module, where $Q_0 = C_G(H) / C_N(H)$. 
This follows directly from the fact that $D$ from Theorem \ref{centralser} is finitely generated as a $\BZ Q_0$-module.

\end{proof}

\noindent 
In contrast to the metabelian case
 there is no criterion for finite presentability of nilpotent-by-abelian groups.  There are, however,  results giving sufficient conditions implying  finite presentability and hence type $\FP_2$. The following results are some of these sufficient conditions.

\begin{proposition}\label{grovesfp} \cite[Cor.~C]{groves}
Let $G$ be a finitely generated nilpotent-by-abelian group with normal nilpotent subgroup $N$ of nilpotency length $c$ with $G/ N$ abelian.  Suppose 
$N/N'$ is $(c+1)$-tame as a $G/N$-module. Then $G$ is finitely presented.
\end{proposition}

\begin{proposition} \label{grovesfp2} \cite[Thm.~B]{groves}
Let $G$ be a finitely generated nilpotent-by-abelian group with normal nilpotent subgroup $N$ of nilpotency length $c$ with $Q = G/ N$ abelian.  Suppose that $\gamma_i(N) / \gamma_{i+1}(N)$ is a finitely generated $\BZ Q$-module and  that $\Sigma_{N / N'}^c(Q) \cap - \Sigma^c_{\gamma_i(N) / \gamma_{i+1}(N)} (Q) = \emptyset$
for every $ 1 \leq i \leq c$.  Then $G$ is finitely presented.
\end{proposition}

\noindent  The following result together with the remarks after Conjecture \ref{Bredonfpdconj} shows that Conjecture \ref{Bredonfpdconj} holds for $m = 2$.
  
  \begin{corollary}\label{centralfp}
Let $G$ be a finitely generated nilpotent-by-abelian group of finite Pr\"ufer rank.  Let $N$ be nilpotent of class $c$ as in (\ref{one}) with $Q$ abelian. Let $H$ be a finite group of order $s$ acting on $G$.  Assume also that $G$ is  of type $\FP_n$ such that $\frac{n}{cs} \geq 2$.   Then $C_G(H)$ is finitely presented. In particular, $C_G(H)$  is of type $FP_2$.
\end{corollary}
  
  \begin{proof}
  By Corollary \ref{centralfg} $C_G(H)$ is finitely generated. By Theorem \ref{centralser}
  $D = \oplus_i \gamma_i(C_N(H)) / \gamma_{i+1} (C_N(H))$ is $\lfloor\frac{n}{sc}\rfloor$-tame as a $C_G(H)/C_N(H)$-module, in particular is 2-tame as a $C_G(H)/C_N(H)$-module. Then by Proposition \ref{grovesfp2} $C_G(H)$ is finitely presented.
  \end{proof}
 
\noindent Theorem \ref{centralser}, its corollaries and Groves' results now lead us to make the following conjecture:

\begin{conjecture}\label{centralserconj}
Let $G$ be a finitely generated nilpotent-by-abelian group. Let $N$ be nilpotent as in (\ref{one}) with $Q$ abelian and suppose  that the direct sum of all factors in the lower central series of $N$ is finitely generated and $m$-tame as a  $G/N$-modules. Then $G$ is of type $\FP_m.$
\end{conjecture}

\noindent This conjecture obviously implies Conjecture \ref{Bredonfpdconj}.  Just combine Proposition \ref{conjclasses}, Theorem \ref{centralser} and Corollary \ref{centralfg} with Theorem \ref{bredonfpn}.

\noindent Also note that Conjecture \ref{centralserconj} is also true if the Hirsch length  $h(G/N) \leq m-1$.  For in this case a result by H. Meinert \cite{meinert} implies that $G$ is of type $\FP_{\infty}.$
Furthermore, by \cite[Theorem 3.13]{martineznucinkis} centralisers of finite subgroups in soluble groups of type $\FP_{\infty}$ are also of type $\FP_{\infty}.$

\noindent The sufficient conditions for finite presentability of $G$ given in Proposition \ref{grovesfp} show that even in the case when $m = 2$ and $N$ is nilpotent of class 2 the converse of Conjecture \ref{centralserconj} is not true. Indeed if $A = N / N'$ is 3-tame then $G$ is finitely presented but if $N' = A \wedge A$ then $N'$ does not need be 2-tame. Actually if $A$ is 4-tame then $N'$ is 2-tame but in general this does not hold if $A$ is only 3-tame. 

\noindent It is a well known fact, see \cite[Proposition 5.3]{BieriGroves}, that for nilpotent-by-abelian groups of type $\FP_m$ the group $H_t(N,\Z)$ is finitely generated for all $1\leq t\leq m$.
Here we show that this condition holds under the assumptions of Conjecture \ref{centralserconj}.

\begin{proposition} Suppose that $G$ is a nilpotent-by-abelian group with derived subgroup $N$ of nilpotency  class $c$, $Q = G/ N$ abelian and assume
that  $D = \oplus_i \gamma_i(N) / \gamma_{i+1}(N)$  is finitely generated and $n$-tame as a $\BZ Q$-module and $n \geq c$. Then $H_t(N,\Z)$ is finitely generated and 
$\lfloor\frac{n}{t}\rfloor$-tame as a $\BZ Q$-module for $1\leq t\leq n $.
\end{proposition}

\begin{proof} 

Denote by  $A_i$  the $i$-th factor in  the
lower central series of $N$. 
The proof of the proposition depends on the structure of $H_j(A_i, \BZ)$. For any torsion-free abelian group $A$ there is a natural isomorphism $H_j(A, \BZ) \simeq \wedge^j A$ \cite[Ch.~V,~Thm.~6.4]{Brown}. The situation turns a bit more complicated when $A$ has torsion. By the proof of \cite[Thm.~C]{Koch2} if $Q$ is a finitely generated abelian group and $A$ is a finitely generated $\BZ Q$-module  then $H_i(A, \BZ)$ has a finite filtration with factors isomorphic to $\BZ Q$-subsections of, possibly different, tensor powers of $A$, where the action of $Q$ is the diagonal one. 
Inspection of the proof of \cite[Thm.~C]{Koch2} shows that the tensor powers are of type $\otimes^j A$ for $j \leq i$.

\noindent
By repeating a Lyndon-Hochschild-Serre
spectral sequence argument one proves that $H_t(N,\Z)$ has a series
with factors which are subsections of modules of the form
$$H_{i_1}(A_1,\Z)\otimes \ldots \otimes H_{i_c}(A_c, \BZ)$$
such that $i_1+ i_2 + \ldots + i_c=t$.
Note that every $H_{i_j}(A_j, \BZ)$ has a filtration with quotients that are subsections of $\otimes^s A_j$ for some $s \leq i_j$  . Then 
$H_{i_1}(A_1,\Z)\otimes\ldots\otimes H_{i_c}(A_c,\Z)$ has a filtration with quotients that are subsections of 
$(\otimes^{s_1} A_1) \otimes (\otimes^{s_2} A_2) \otimes \ldots \otimes (\otimes^{s_c} A_c)$ for $s_j \leq i_j$, hence these quotients are subsections of $\otimes^{s_1 + s_2 + \ldots + s_c} D$ for  $s_1 + s_2 + \ldots + s_c \leq 
i_1+ i_2 + \ldots + i_c=t$.
Since $H_t(N, \BZ)$ has a filtration with quotients that are subsections of $\otimes^m D$ for $m \leq t \leq n$ and $\otimes^i D$ is finitely generated as a $\BZ Q$-module for $i \leq n,$ we deduce that $H_t(N, \BZ)$ is finitely generated as a $\BZ Q$-module for $t \leq n$.

\noindent
Finally \cite[Lemma 1.1. and Theorem 1.3.]{bs81} yield
 $$\begin{aligned}\Sigma^c_{H_t(N,\Z)}(Q)&\subseteq\bigcup_{s \leq t}\Sigma^c_{\otimes^{s}D}(Q)
 \subseteq\\
 &\bigcup_{s \leq t}\text{conv}_{ \leq s}\Sigma^c_D(Q)\subseteq
 \text{conv}_{\leq t}\Sigma^c_D(Q).
 \end{aligned}$$
As $0\not\in\text{conv}_{\leq n}\Sigma^c_D(Q)$, we deduce
$$0\not\in\text{conv}_{\leq \frac{n}{t}}\Sigma^c_{H_t(N,\Z)}(Q).$$

\end{proof}



\end{document}